\documentclass[12pt]{amsart}
\usepackage{amssymb}
\usepackage{amscd}
\usepackage[latin1]{inputenc}

\usepackage[active]{srcltx}

\nonstopmode

\newcommand {\PP}{\mathbb{P}}
\DeclareMathOperator{\codim}{codim}

\DeclareMathOperator{\pnt}{\raise 0.5mm \hbox{\large\bf.}}

\newtheorem{theorem}{Theorem}[section]
\newtheorem{lemma}[theorem]{Lemma}

\newtheorem{corollary}[theorem]{Corollary}

\theoremstyle{definition}
\newtheorem{remark}[theorem]{Remark}

\newtheorem{problem}[theorem]{Problem}

\begin{document}

\title{Glicci ideals}

\author[Juan Migliore]{Juan Migliore${}^*$}
\address{
Department of Mathematics \\
 University of Notre Dame \\
  Notre Dame,
IN 4655\\
 USA}
\email{Juan.C.Migliore.1@nd.edu}
\author[Uwe Nagel]{Uwe Nagel${}^{+}$}
\address{Department of Mathematics \\
University of Kentucky \\
715 Patterson Office Tower \\ Lexington, KY 40506-0027\\
 USA}
\email{uwenagel@ms.uky.edu}

\begin{abstract}
  A central problem in liaison theory is to decide whether every arithmetically Cohen-Macaulay subscheme of projective $n$-space can be linked by a finite number of arithmetically  Gorenstein schemes to a complete intersection. We show that this can be indeed achieved if the given scheme is also generically Gorenstein and we allow the links to take place in an $(n+1)$-dimensional projective space. For example, this result applies to all reduced arithmetically Cohen-Macaulay subschemes. We also show that every union of fat points in projective 3-space can be linked in the same space to a union of simple points in finitely many steps, and hence to a complete intersection in projective 4-space.
\end{abstract}

\thanks{${}^*$ The work for this paper was done while the first
author was sponsored by the National Security Agency under Grant
Numbers H98230-09-1-0031 and H98230-12-1-0204, and by the Simons Foundation under grant \#208579.\\
${}^+$ The work for this paper was done while the second author was
sponsored by the National Security Agency under Grant
Numbers H98230-09-1-0032 and H98230-12-1-0247, and by the Simons Foundation under grant \#208869.
}

\maketitle

\section{Introduction}
\label{sec-intro}

A central problem is liaison theory has always been to describe the ideals/subschemes that are linked to a complete intersection in a finite number of steps. When the links are all by complete intersections (CI-liaison) these are named {\em licci ideals}, whereas when the links are Gorenstein (G-liaison) they are dubbed {\em glicci ideals}. It is well-known that each glicci (and in particular licci) ideal is Cohen-Macaulay. However, the converse is false for CI-liaison. There are many arithmetically Cohen-Macaulay schemes that are not licci, for instance a set of four points in $\mathbb P^3$ or a fat point in $\mathbb P^3$ (see \cite{HU-Annals} and \cite{PU}, respectively).  These two examples are, however, glicci.
In \cite{KMMNP} the question was raised of whether the converse is true in general for G-liaison, i.e., whether every Cohen-Macaulay  ideal is glicci. While there are many affirmative partial results in this direction (see, e.g., \cite{KMMNP}, \cite{CDH}, \cite{Compositio}, \cite{HSS}, \cite{Gorla},  \cite{NR}), there has also been scepticism. Most notably, \cite{Hartshorne} and \cite{HSS} discuss certain sets of points which cannot be shown to be glicci with current methods (see Remark \ref{rem:Hartshorne-exa}). In this note we propose to shift focus by enlarging the ambient space of the embedded subschemes.  Henceforth, ``link" will be understood to mean by Gorenstein ideals (but recall that in codimension two, all Gorenstein ideals are complete intersections).

Let $X \subset \PP^n$ be a closed subscheme. Considering $\PP^n$ as a hyperplane of some $\PP^{n+1}$, we can view $X$ as a subscheme of $\PP^{n+1}$. If $X$ is any finite set of points, then we will see that $X$ is glicci as a subscheme of $\PP^{n+1}$. In fact, much more is true. A special case of our Theorem \ref{thm-CM-with-G0} below is:

\begin{theorem}
  \label{thm-red-lift}
Let $X \subset \PP^n$ be a reduced arithmetically Cohen-Macaulay subscheme.  Considering $X$ as a subscheme of $\PP^{n+1}$, $X$ can be linked in $\PP^{n+1}$ to a complete intersection in $\PP^{n+1}$.
\end{theorem}

This motivates the following:

\begin{problem}
  \label{prob-contains-red}
{\em Let $X \subset \PP^n$ be an arithmetically Cohen-Macaulay subscheme. Can $X$ be G-linked in $\mathbb P^n$ to a reduced subscheme in a finite number of steps?  Is it at least true if we allow the links to take place in $\mathbb P^{n+1}$?}
\end{problem}

Of course, combined with Theorem \ref{thm-red-lift},  an affirmative answer would imply that every arithmetically Cohen-Macaulay subschema of $\PP^n$ can be linked to a complete intersection in some higher-dimensional projective space.

Another  motivation for the above problem is the result of Rao \cite{rao-invent} that says that every even liaison class of curves in $\mathbb P^3$ contains a smooth element.  On the other hand, in higher codimension we have the striking result of Polini and Ulrich \cite{PU} that says that in some situations, there is a non-reduced locus that is contained in every element of the CI-liaison class.  For example, if $Z$ is any  union of (non-reduced) fat points in $\mathbb P^n$, then $Z$ is a subscheme of every element of its even liaison class, which hence contains no reduced element.  Nevertheless, it was shown in \cite{KMMNP} that a single fat point (as a very special case of a deeper result) can be G-linked in a finite number of steps to a single point.  This result will be generalized in this paper to an arbitrary union of fat points in $\mathbb P^3$, even of mixed multiplicities.

The technique that proves Theorem \ref{thm-red-lift} is not limited to reduced subschemes.
As example, we discuss monomial ideals. In \cite{Compositio}, we showed that strongly stable artinian ideals are glicci. Using a similar technique, Huneke-Ulrich \cite{HU} proved that every monomial artinian ideal is glicci in the localization with respect to the homogeneous maximal ideal. This result is not useful for applications when one is interested in information about the Hilbert function. Thus, it is still an important open problem whether each such ideal can be linked homogenously to a complete intersection. Again, we have an affirmative answer if we pass to a larger polynomial ring.

\begin{theorem}
  \label{thm-monomial-ideal}
Let $I \subset R := K[x_0,\ldots,x_n]$ be a monomial  ideal such that $R/I$ is Cohen-Macaulay. Let $t$ be a new variable. Then the ideal $(I, t) \subset R[t]$ is glicci in $R[t]$.
\end{theorem}

Since $R/I \cong R[t]/(I, t)$, this result should be as useful for applications as an affirmative answer to the original question in $R$ would be.

We  remark that analogous results hold for ideals in regular local rings.


\section{Reduced schemes}

Let $R = K[x_0,\ldots,x_n]$ be a polynomial ring over a field $K$ with its  standard grading.  The basis of our results is the following observation.

\begin{lemma}
  \label{lem-key}
Let $I \subset R$ be a homogeneous ideal such that $R/I$ is Cohen-Macaulay and generically Gorenstein.  If $f \in R$ is a homogeneous nonzero divisor of $R/I$, then the ideal $(I, f) \subset R$ is glicci.
\end{lemma}

\begin{proof}
The assumptions on $A := R/I$ guarantee that its canonical module $K_A$ can be identified with an ideal of $A$ (cf., for example, \cite{BrHe}, Proposition 3.3.18). It follows that there is a Gorenstein ideal $J \subset R$ containing $I$ such that $\codim J = \codim I + 1$ and $K_A (- t) \cong J/I$ for some integer $t$. We claim that the ideal $(I, f)$ is directly G-linked to $J$.

Indeed, by \cite[Lemma 4.8]{KMMNP}, the ideal $I + f J$ also is a Gorenstein ideal with the same codimension as $J$. Thus it suffices to show that
\begin{equation*}
  \label{eq-ideal-quotient}
(I + f J) : (I, f) = J.
\end{equation*}
However, since $I : f = I$ this follows easily because
\[
(I + f J) : (I, f) = (I + f J) : f = (I : f) + J = J.
\]
(The second equality follows by a standard argument.)
We have shown that $(I, f)$ is directly G-linked to the Gorenstein
ideal  $J$. Therefore $(I, f)$ is glicci because every Gorenstein
ideal is glicci by \cite{CDH}. Note that the proof in \cite{CDH} is for subschemes of $\PP^n$. However, by using graded modules and local cohomology instead of sheaves, the arguments can be adapted so that they also apply to artinian ideals.
\end{proof}

An immediate consequence is the following:

\begin{corollary} \label{cor of lem-key}
Let $V_1$ be a generically Gorenstein, arithmetically Cohen-Macaulay subscheme of $\mathbb P^n$ and let $V_2$ be a complete intersection that meets $V_1$ properly, that is, $\dim ( V_1 \cap V_2) = \dim V_1 + \dim V_2 - n$. Then $V_1 \cap V_2$ is glicci as a subscheme of $\mathbb P^n$.
\end{corollary}

Our first main result is a generalization of Theorem \ref{thm-red-lift}.

\begin{theorem}
  \label{thm-CM-with-G0}
Let $X \subset \PP^n$ be an arithmetically Cohen-Macaulay subscheme that is generically Gorenstein.   Then, considering $X$ as a subscheme of $\PP^{n+1}$, $X$ can be linked in $\PP^{n+1}$ to a complete intersection of $\PP^{n+1}$.
\end{theorem}

\begin{proof}
Let $t$ be a new variable so that $S := R[t]$ is the coordinate ring of $\PP^{n+1}$. By assumption, $R/I_X$ is Cohen-Macaulay and generically Gorenstein, thus so is $S/ I_X S$. Hence, Lemma \ref{lem-key} provides that $(I_X S , t) \subset S$ is glicci in $S$. Since $S/(I_X S , t) \cong R/I_X$ this concludes the argument.
\end{proof}

Notice that this result establishes Theorem \ref{thm-red-lift} because every reduced subscheme is generically Gorenstein.

\begin{remark}
  \label{rem:Hartshorne-exa}
In \cite{Hartshorne}, Hartshorne investigated the problem of whether arithmetically Cohen-Macaulay subschemes are glicci in the case of general sets of points in $\PP^3$. He showed that the answer is affirmative if the number of points is at most 19 and proposed a set of 20 general points as a possible counterexample. The problem has been studied further in \cite{HSS}, where it is shown that every general set of at least 56 points in $\PP^3$ cannot be linked to a complete intersection by strictly descending Gorenstein liaison or biliaison. In contrast, our Theorem \ref{thm-red-lift} guarantees that every set of points in projective space can be linked to a complete intersection if we pass to a higher-dimensional space.
\end{remark}

Now we turn to considering monomial ideals that are not necessarily generically Gorenstein.

\begin{proof}[Proof of Theorem \ref{thm-monomial-ideal}] \mbox{} \\
Again, let $t$ be a new variable, and set $S := R[t]$. By \cite{lift}, we can lift the ideal $I \subset R$ to a reduced ideal $J \subset S$ such that $S/(J, t) \cong R/I$ and $J : t = J$. In particular, $S/J$ is Cohen-Macaulay. Hence,  Lemma \ref{lem-key} yields that $(J, t)$ is glicci in $S$, and our claim follows because $(J, t) = (I, t) S$.
\end{proof}


\section{Unions of fat points}

In this section we will give an affirmative answer to Problem \ref{prob-contains-red} in the case where $X$ is a union of fat points in $\mathbb P^3$ (not necessarily of the same degree).  We recall that if $P$ is a point in $\mathbb P^n$ with saturated ideal $\mathfrak p$ then a {\em fat point} supported at $P$ is a zero-dimensional scheme defined by the saturated ideal $\mathfrak p^k$ for some $k \geq 1$.  Of course if $k >1$ then the fat point is non-reduced. Since a fat point is defined by a standard determinantal ideal, it is glicci by the Gaeta-type result in \cite{KMMNP}. However, we begin by giving a new proof for this result, which illustrates the arguments for our more general result in a particularly transparent manner. It uses ideas from \cite{MN-Adv}.  Recall the the $h$-vector of a $d$-dimensional arithmetically Cohen-Macaulay subscheme is the $d+1$-st difference of its Hilbert function.

\begin{lemma}
  \label{lem:one-fat-point}
A fat point of $\PP^n$ is glicci.
\end{lemma}

\begin{proof}
Let $P \in \mathbb P^n$ be a point, and without loss of generality assume that $I_P = \langle x_1,\ldots,x_n \rangle$.  Let $Z \subset \PP^n$ be the fat point defined by the saturated ideal $\langle x_1,\ldots,x_n \rangle ^a$. It is well-known that the $h$-vector of $Z$ is
\[
\left (1,n,\binom{n+1}{2},\dots,\binom{n+a-2}{a-1} \right )
\]
By \cite[Theorem 4.3]{MN-Adv}, taking $2 a + n-2$ general linear forms in $I_P$, one can define an arithmetically Gorenstein curve $G \subset \PP^n$  that is a union of $2 \binom{n+a-2}{a-1}$ lines through $P$ with $h$-vector
\[
\left  (1,n-1,\binom{n}{2},\dots,\binom{n+a-3}{a-1}, \binom{n+a-3}{a-1},\ldots,\binom{n}{2}, n-1, 1 \right  ).
\]
Comparing with \cite[Theorem 4.1]{MN-Adv}, the curve $G$ contains an arithmetically Cohen-Macaulay subcurve $C$ consisting of $\binom{n+a-2}{a-1}$ lines with $h$-vector
\[
\left  (1,n-1,\binom{n}{2},\dots,\binom{n+a-3}{a-1} \right  ).
\]
Note that both, $G$ and $C$, are cones over suitable sets of points in the hyperplane $\{x_0 = 0\} \cong \PP^{n-1}$.

\bigskip

\noindent {\bf Claim.}
$Z \subset C$.

\medskip

We want to show $I_C \subset I_Z = I_P^a = (x_1,\ldots,x_n)^a$.  From the $h$-vector we see that $I_C = \langle F_1,\dots,F_{\binom{n+a-2}{a}}\rangle$, where the $F_i$ are homogeneous polynomials of degree $a$ in the variables $x_1,\ldots,x_n$. The latter is true by construction of $C$. Thus all the $i$th partial derivatives of the $F_j$, for $1 \leq i \leq a-1$, vanish at $P$.  Hence $I_C \subset I_Z$ as claimed.

\bigskip

Let $D$ be the complement of $C$ in $G$. This means that the curves $C$ and $D$ are geometrically linked by $G$. Furthermore, the curve $D$ is also a reduced union of lines through $P$, having the same $h$-vector as $C$.  In the same way as above, $Z \subset D$.

Hence $I_C + I_{D}$ defines an arithmetically Gorenstein zero-dimensional scheme $X$, supported at $P$.  Using \cite[Lemma 2.5]{MN-Adv}, we find that the $h$-vector of $X$ is:
\[
\begin{array}{cccccccccccccccccccccc}
\hbox{deg} & 0 & 1 & 2 & \dots & a-2 & a-1 & a & \dots & 2a-4 & 2a-3 & 2a-2 \\
\hbox{$h$-vector} & 1 & n & \binom{n+1}{2} & \dots & \binom{n+a-3}{a-2} & \binom{n+a-2}{a-1} & \binom{n+a-3}{a-2} & \dots & \binom{n+1}{2} & n & 1

\end{array}
\]
Thus $X$ links $Z$ to a zero-dimensional scheme $Z'$ supported at $P$ with $h$-vector
\[
\left (1,n,\binom{n+1}{2},\dots,\binom{n+a-3}{a-2} \right ).
\]
By construction, the generators of $I_X$ are polynomials only in $x_1,\ldots,x_n$.  Thus $I_{Z'} = I_P^{a-1}$.

Continuing inductively, we see that this construction provides a series of links from $Z$ down to $P$ itself, which is reduced.  This completes the proof.
\end{proof}

If $n=3$, the constructions in the above proof become a lot easier. In fact, then the above arithmetically Gorenstein scheme $G$ is simply a complete intersection of forms of degree $a$ and $a+1$, respectively, both of which are products of linear forms in $I_P$. This simplicity is the reason why we restrict ourselves to the case $n=3$ for showing  how to pass from the case of one fat point to a union of fat points.

\begin{theorem} \label{fat its}
Let $P_1,\dots,P_s $ be points in $\mathbb P^3$, with defining ideals $\mathfrak p_1 ,\dots, \mathfrak p_s$ respectively.  Choose positive integers $b_1 ,\dots, b_s$ and let $Z$ be the scheme defined by the saturated ideal $\mathfrak p_1^{b_1}  \cap \dots \cap \mathfrak p_s^{b_s}$.  Then $Z$ is G-linked in $\mathbb P^3$ to a reduced scheme in a finite number of steps.
\end{theorem}

\begin{proof}
In this proof we will abuse notation somewhat and use the same name for a plane and for the linear form defining it.

Our approach will be to apply a series of {\em evenly many} Gorenstein links with the following strategy.  We will focus on one point at a time, and {\em locally at that point} we will be reducing it via the links described in Lemma \ref{lem:one-fat-point}  for $s=1$.  However, globally we will leave any other existing points the way they are (i.e. any reduced points stay reduced, and any fat points retain their multiplicity), and will will pick up new points, but they will all be reduced. Thus at the end we will have removed any non-reduced structure on the original points, and picked up {\em many} new reduced points, but we will have produced a reduced scheme in the even Gorenstein class of $Z$.

For compatibility with Lemma \ref{lem:one-fat-point}, set $b_1 = a$, $P_1 = P$, and $I_P = \mathfrak p$.  Let $Z$ be the scheme defined by $I_Z = \mathfrak p^a \cap \mathfrak p_2^{b_2} \cap \dots \cap \mathfrak p_s^{b_s}$, where $a \geq 2$ and $b_i \geq 1$ for $2 \leq i \leq s$.  Let
\[
\begin{array}{ccl}
A & = & \hbox{product of $a$ general linear forms in $\mathfrak p$} \\
B & = & \hbox{product of $a+1$ general linear forms in $\mathfrak p$}
\end{array}
\]
The ideal $\langle A,B \rangle$ is a complete intersection consisting of $a^2+a$ lines through $P$.  As pointed out above, there is a subset $C$ consisting of $\binom{a+1}{2}$  of these lines with $h$-vector $(1,2,3, \dots, a)$. Note that the curve $C$  is arithmetically Cohen-Macaulay (being a cone over a finite set of points).  Let $D$ be the residual in the complete intersection defined by $\langle A,B \rangle$.  Notice that $D$ is also arithmetically Cohen-Macaulay and has the same $h$-vector as $C$.  By the claim above, we have $I_C \subset \mathfrak p^a$ and $I_D \subset \mathfrak p^a$.

Choose general planes $L_{i,j} \in \mathfrak p_i$, where $2 \leq i \leq s$ and $1 \leq j \leq b_i$.  Notice that $\prod_j L_{i,j} \in \mathfrak p_i^{b_i}$.  Let
\[
F = A \cdot \prod_{i,j} L_{i,j}.
\]
Notice that $\deg F = a + \sum_{i=2}^s b_i$ is the sum of the powers of the ideals of the points, and that $F \in I_C$ (hence also in $I_Z$).

Choose general planes $M_{i,j} \in \mathfrak p_i$, where $2 \leq i \leq s$ and $1 \leq j \leq b_i$.  Let
\[
Q = \prod_{i,j} M_{i,j}.
\]
Define a curve $Y$ with defining ideal
\[
I_Y = Q \cdot I_C + \langle F \rangle.
\]
This is a basic double link (see for instance \cite{migbook}), and since $C$ is arithmetically Cohen-Macaulay, it follows that so is $Y$. Furthermore, by the genericity of $L_{i, j}$ and $M_{i, j}$ and comparing degrees, we  see that
\[
I_Y = I_C \cap \langle F, Q \rangle.
\]
Thus,
$Y$ is a union of lines with at most two passing through any point other than the $P_i$ or $P$.  It also follows that $Z \subset Y$.
Through each point $P_i$ ($2 \leq i \leq s$) pass $b_i^2$ lines of $Y$, namely the {\em complete intersection} of $\prod_j L_{i,j}$ and $\prod_j M_{i,j}$. The components of $Y$ containing $P$ form the curve $C$.

Choose general planes $N_{i,j} \in \mathfrak p_i$, where $2 \leq i \leq s$ and $1 \leq j \leq b_i$.  Let
\[
G = QB \cdot \prod_{i,j} N_{i,j}.
\]
We make the following observations about the ideal $\langle F,G \rangle$.

\begin{enumerate}

\item It is a complete intersection defining a union of lines.

\item An any point other than $P$ or one of the $P_i$, no more than two of these  lines pass.

\item At any of $P_i \in \{P_2,\dots, P_s\}$, locally it is a complete intersection of type $(b_i, 2b_i)$, namely the complete intersection of $\prod_j L_{i,j}$ and $\prod_j M_{i,j} \cdot \prod_j N_{i,j}$.

\item At $P$ it is a complete intersection of type $(a,a+1)$, namely that defined by $\langle A,B \rangle$.

\item $F \in I_Y$ by construction and $G \in I_Y$ since $B \in I_C$ so $QB \in Q \cdot I_C$; hence $\langle F,G\rangle \subset I_Y$.

\end{enumerate}

Thus $\langle F,G \rangle$ links $Y$ to a residual curve $W$, which is arithmetically Cohen-Macaulay since $Y$ is.  $W$ is again a union of lines, with at most two meeting at any point away from $P$ or one of the $P_i$.  Locally at $P$, the components of $W$ coincide with the components of $D$.  At any $P_i \in \{ P_2,\dots,P_s \}$, $W$ is a complete intersection of type $(b_i,b_i)$; namely it is the complete intersection of $\prod_j L_{i,j}$ and  $\prod_j N_{i,j}$.

By a similar reasoning as above, $Z \subset W$.  Furthermore, the link joining $Y$ and $W$ is a geometric link.  Thus by standard liaison theory (see e.g. \cite{migbook}), $I_Y + I_W$ is the saturated ideal of a codimension three Gorenstein scheme, which we will call $Gor$; and since $I_Y \subset I_Z$ and $I_W \subset I_Z$, we obtain $Z \subset Gor$.

Locally at any $P_i \in \{ P_2,\dots,P_s \}$, $Gor$ is the complete intersection
\[
\left \langle \prod_j L_{i,j}, \prod_j M_{i,j}, \prod_j N_{i,j} \right \rangle.
\]
Locally at $P$, $Gor$ is precisely the Gorenstein scheme $X$ described in the proof of Lemma \ref{lem:one-fat-point} in the case $n=3$.  Everywhere else, $Gor$ is reduced.

Since $I_{Gor} \subset I_Z$, $Gor$ links $Z$ to a residual zero-dimensional scheme $Z'$.  Thanks to the work done in the case $s=1$, the component of $Z'$ supported at $P$ is just defined by $\mathfrak p^{a-1}$.  The components supported at $P_2,\dots, P_s$ are non-reduced in general (and it will not matter to us what they look like there), but everywhere else $Z'$ is reduced.  Denote by $\{R_k\}$ the remaining points other than $P$ and $P_2,\dots,P_s$. It does not matter how many points $R_k$ there are; let us suppose $1 \leq k \leq \tau$.

Now we will perform a second link, from $Z'$ to a new zero-dimensional scheme $Z''$ which will have important properties for us.  The approach is similar, but with small differences.
Let
\[
\begin{array}{rcl}
A' & = & \hbox{product of $a$ general linear forms in $\mathfrak p$} \\
B' & = & \hbox{product of $a-1$ general linear forms in $\mathfrak p$} \\
\end{array}
\]
Then $\langle A',B' \rangle$ is a complete intersection defining a reduced union of $a^2 - a$ lines through $P$.  Let $C'$ be a subset of $\binom{a}{2}$ of these lines with $h$-vector $(1,2,3,\dots,a-1)$.  $C'$ is an arithmetically Cohen-Macaulay curve.   Let $D'$ be the residual.  $D'$ is also arithmetically Cohen-Macaulay, with the same $h$-vector.  As earlier, $I_{C'} \subset \mathfrak p^{a-1}$ and $I_{D'} \subset \mathfrak p^{a-1}$.

For each $1 \leq k \leq \tau$, let $L_k', M_k'$ and $N_k'$ be general linear forms vanishing on $R_k$.  Let
\[
\begin{array}{rcl}
F' & = & A' \cdot \prod_{i,j} L_{i,j} \cdot \prod_{k} L_k' \\ \\
Q' & = & Q \cdot \prod_k M_k' \\ \\
I_{Y'} & = & Q' \cdot I_{C'} + \langle F' \rangle
\end{array}
\]
As before, $Y'$ is an arithmetically Cohen-Macaulay union of lines.

The components of $Y'$ passing through $P$ are precisely $C'$.  The components of $Y'$ passing through $P_2,\dots,P_s$ form a complete intersection of lines (namely at $P_i$ they are defined by $\langle \prod_j L_{i,j}, \prod_j M_{i,j} \rangle$).  Through each $R_k$ there is just one line (namely that defined by $\langle L_k', M_k' \rangle$).

Now let
\[
G' = Q' \cdot B' \cdot \prod_{i,j} N_{i,j} \cdot \prod_k N_k'.
\]
Since both $F'$ and $G'$ are in $I_{Y'}$, and by the genericity of their choices, $\langle F', G' \rangle$ provides a geometric link of $Y'$ to a residual curve, $W'$.  Locally at $P$, $W'$ consists of the components of $D'$.  Locally at $P_i$, $2 \leq i \leq s$, $W'$ consists of the same complete intersection that we had for $W$, namely the complete intersection of type $(b_i,b_i)$ given by $\prod_j L_{i,j}$ and  $\prod_j N_{i,j}$. Locally at $R_k$, $W'$ is the complete intersection of type $(1,1)$ (i.e. a single line) given by $ \langle L_k, N_k \rangle$.

Now, since $Y'$ and $W'$ are geometrically linked, the sum of their ideals is an arithmetically Gorenstein ideal, $Gor'$. At $P$ it is the intersection of the linked arithmetically Cohen-Macaulay curves $C'$ and $D'$, hence the component of $Gor'$ at $P$ is itself arithmetically Gorenstein, with $h$-vector
\[
\left ( 1,3,6,\dots, \binom{a-1}{2}, \binom{a}{2}, \binom{a-1}{2}, \dots, 6,3,1 \right ).
\]
The component of $Gor'$ supported at $P_i$ is the complete intersection
\[
\langle \prod_j L_{i,j}, \prod_j M_{i,j} \prod_j N_{i,j} \rangle.
\]
The component of $Gor'$ supported at $R_k$ is reduced.  There are also other reduced components of $Gor'$.

Now, $Gor'$ provides a link of $Z'$ to a residual zero-dimensional scheme $Z''$.  The component of $Z''$ supported at $P$ has defining ideal $\mathfrak p^{a-2}$ thanks to the calculations in the proof of Lemma \ref{lem:one-fat-point}.  The component of $Z''$ at $P_i$, $2 \leq i \leq s$, are the original fat points, $\mathfrak p_i^{b_i}$, since locally we have linked twice using the same complete intersection.  $Z''$ has no components supported at any of the $R_k$.  Everywhere else, $Z''$ is reduced.

Thus in two links we have passed from the ideal $\mathfrak p^a \cap \mathfrak p_2^{b_2} \cap \dots \cap \mathfrak p_s^{b_s}$ to a zero-dimensional scheme consisting of the form
\[
\mathfrak p^{a-2} \cap \mathfrak p_2^{b_2} \cap \dots \cap \mathfrak p_s^{b_s} \cap J
\]
where $J$ is the ideal of a reduced set of points. Thus, depending on whether $a$ is even or odd, we can reduce in a finite (even) number of Gorenstein links to a zero-dimensional scheme where the components supported at  $P_2,\dots, P_s$ are the original fat points, the component at $P$ is either reduced or empty, and all other points are reduced.  We can then do the same thing at each of $P_2,\dots,P_s$, obtaining in the end a reduced zero-dimensional scheme (of very large degree).
\end{proof}

In the above result it is crucial that we are not bound to use complete intersections for the links. In fact, by the results in \cite{PU} it not possible to link any union of fat points that is not reduced to a union of simple points by using only complete intersections for the links.

\begin{corollary}
  \label{cor:fat-points}
Every union of fat points in $\PP^3$ can be linked in $\PP^4$ to a complete intersection of $\PP^4$ in a finite number of steps.
\end{corollary}

\begin{proof}
Combine Theorems \ref{fat its} and \ref{thm-CM-with-G0}.
\end{proof}

Using the arguments for establishing Lemma \ref{lem:one-fat-point}, we believe that the ideas in the proof of Theorem \ref{fat its} can be extended to show that every union of fat points in $\PP^n$ can be linked to a reduced subscheme of $\PP^n$. However, the details seem to become very demanding.


\end{document}